\documentclass[reqno, 12pt]{amsart}
\usepackage{extsizes}
\usepackage{amsmath}
\usepackage{amssymb}
\usepackage[dvipsnames]{xcolor}
\usepackage{amsthm}
\usepackage[normalem]{ulem}
\usepackage[utf8]{inputenc}
\usepackage{indentfirst}
\usepackage{mathtools}
\usepackage{esint}
\usepackage{pdfpages}
\usepackage[hidelinks]{hyperref}
\usepackage{enumitem}
\usepackage{pgf, tikz}
\usetikzlibrary{arrows.meta}
\usepackage{caption}
\usepackage{subcaption}
\usepackage{pgfplots}
\pgfplotsset{compat=1.18}
\usepackage{adjustbox}
\usepackage{verbatim}
\usepackage{hyperref}
\def\hmath$#1${\texorpdfstring{{\rmfamily\textit{#1}}}{#1}}
\usepackage{bbm}
\usepackage{dsfont}
\usepackage{setspace}
\usepackage[margin=3cm]{geometry}
\usetikzlibrary{tikzmark}
\usepackage{tzplot}
\usepackage{cancel}

\usepackage{graphicx}
\setlist{nosep}
\hypersetup{
    colorlinks,
    citecolor=MidnightBlue,
    filecolor=black,
    linkcolor=MidnightBlue,
    urlcolor=black,
    linktoc=page
}
\mathtoolsset{showonlyrefs}

\newtheorem{theorem}{Theorem}[section]

\newtheorem{lemma}[theorem]{Lemma}
\newtheorem{proposition}[theorem]{Proposition}

\newtheorem{remark}[theorem]{Remark}

\theoremstyle{definition}

\newtheorem*{timescaling*}{Time-Scaling}

\newcommand{\R}{\mathbb{R}}

\newcommand{\N}{\mathbb{N}}
\newcommand{\Z}{\mathbb{Z}}

\newcommand{\F}{\mathcal{F}}

\newcommand{\norm}[1]{\left\lVert#1\right\rVert}

\newcommand{\inftynorm}[1]{\left\lVert#1\right\rVert_{L^\infty}}

\newcommand{\ltwonorm}[1]{\left\lVert#1\right\rVert_{L^2}}

\newcommand{\lpnorm}[2]{\left\lVert#1\right\rVert_{L^{#2}}}

\newcommand{\abs}[1]{\left\lvert#1\right\rvert}

\DeclareMathOperator\supp{supp}

\newcommand{\absxi}{\abs{\xi}}

\newcommand{\eps}{\varepsilon}

\newcommand{\absnabla}{\abs{\nabla}}

\newcommand{\blue}[1]{\textcolor{black}{#1}}

\makeatletter
\newcommand{\AlignFootnote}[1]{%
    \ifmeasuring@
    \else
        \footnote{#1}%
    \fi
}
\makeatother

\newtagform{blue}[\textcolor{MidnightBlue}]{\color{black}(}{)}

\let\oldtocsection=\tocsection
\let\oldtocsubsection=\tocsubsection

\renewcommand{\tocsection}[2]{\hspace{0em}\oldtocsection{#1}{#2}}
\renewcommand{\tocsubsection}[2]{\hspace{1em}\oldtocsubsection{#1}{#2}}
\usepackage{xpatch}
\makeatletter   
\xpatchcmd{\@tocline}
{\hfil\hbox to\@pnumwidth{\@tocpagenum{#7}}\par}
{\ifnum#1<0\hfill\else\dotfill\fi\hbox to\@pnumwidth{\@tocpagenum{#7}}\par}
{}{}
\makeatother    
\makeatletter
 \def\l@subsection{\@tocline{2}{0pt}{4pc}{6pc}{}}
\def\l@subsubsection{\@tocline{3}{0pt}{8pc}{8pc}{}}
 \makeatother

\allowdisplaybreaks

\numberwithin{equation}{section}

\title[Stratification in the 2D inviscid Boussinesq system]{The effect of \blue{linear} stratification on the stability of a rest state in the 2D inviscid Boussinesq system }
\author[C. Jurja]{Catalina Jurja}
\address{Institute of Mathematics, University of Zurich, Zurich, Switzerland}
\email{catalina.jurja@math.uzh.ch}
\author[H. Ko]{Haram Ko}
\address{Brown University, Providence, RI, USA}
\email{haram\_ko@brown.edu}



\begin{document}
\begin{abstract}
 We investigate and quantify the effect of linear stratification on the stability time of a stably stratified rest state for the 2D inviscid Boussinesq system on $\R^2$. As an important consequence, we obtain stability of the steady state starting from an $\eps$-sized initial perturbation of Sobolev regularity $H^{3^+}$ on a timescale $\mathcal{O}(\eps^{-4/3})$. 

 In our setting, the linear stratification induces dispersion and at the core of our approach lies a suitable bootstrap framework based on Strichartz estimates. This allows us to keep \textit{only $L^2$-based regularity} assumptions on the initial perturbation, whereas previous works impose additional localizations to achieve this timescale. 

 We prove the analogous result for the related dispersive SQG equation.
\end{abstract}
\vspace*{-11pt}
\maketitle
\setcounter{tocdepth}{1}
\vspace{-1cm}
\tableofcontents   
\usetagform{blue}

\begin{section}{Introduction}

The purpose of this paper is to study the stability of a particular steady state of the 2D inviscid Boussinesq system:
\begin{equation}\label{eqn: BQ}
\left\{
      \begin{aligned}
 &\partial_t v +v\cdot \nabla v=-\nabla p -\kappa\varrho\vec{e_2},\\
 &\partial_t\varrho+v\cdot \nabla\varrho=0, \\
    &\mathrm{div} \hspace{1mm} v=0.
\end{aligned}
    \right.
\end{equation}
These equations describe the motion of an incompressible fluid of velocity $v:\R^+\times \R^2\to \R^2$, pressure $p:\R^+\times \R^2\to \R$ and density $\varrho:\R^+\times \R^2\to \R$  under the influence of gravity. The gravity force is chosen to be pointing downwards in the $-\vec{e_2}$ direction with constant of gravity $\kappa>0$.  The system \eqref{eqn: BQ} arises from the Boussinesq approximation of the 3D inhomogeneous Euler equation, see \cite[\textcolor{MidnightBlue}{\S 2.4}]{vallis2017atmospheric}.

Aside from its relevance in geophysics, the system presents mathematical interest due to its link with 3D axisymmetric flows \cite[\textcolor{MidnightBlue}{Ch. 5.4.1}]{majda2002vorticity}. Beyond the local well-posedness theory \cite{LWP_Chae} in $H^{s}, s>2$, and refinements to other regularities \cite{BQ_LWP_Besov, BQ_LWP_Holder}, in general little is known about the long time dynamics of solutions to this system. Some settings may exhibit blow-up \cite{Elgindi-Jeong-SingFormBQ, CH22, CH24} and ill-posedness \cite{bianchini2024strongillposednesslinfty2D} scenarios. Here, we study the stability of the stratified steady state
\begin{equation}\label{eq:steady}
    (v_s,\varrho_s)=(0,-\kappa x_2),\qquad p_s(x_1,x_2)=\kappa x_2^2/2.
\end{equation}
That is, for $v=v_s+u=u,\varrho=\varrho_s+\rho=-\kappa x_2+\rho$ we consider the Cauchy problem
\begin{equation}\label{eqn: perturbed BQ}
\left\{
      \begin{aligned}
 &\partial_t u +u\cdot \nabla u=-\nabla p -\kappa\rho\vec{e_2},\\
 &\partial_t\rho+u\cdot \nabla\rho=\kappa u_2, \\
    &\mathrm{div} \hspace{1mm} u=0,\\
    & u(0,x)=u_0(x),\; \rho(0,x)=\rho_0(x).
\end{aligned}
    \right.
\end{equation}
The steady state \eqref{eq:steady} provides a prototypical setting of a \emph{stably stratified} fluid, where the density of the fluid increases in the direction of gravity (i.e.\ $\varrho_s'(x)<0$), see e.g.\ \cite{gallay2019stabilityvorticesidealfluids}.

It is well-known \cite{Elgindi_2015}, that in the setting of \eqref{eqn: perturbed BQ}, stratification induces dispersion whose strength is captured by  the parameter $\kappa$ and the dispersion relation is given by \eqref{def:dispersive-unknowns}, see Sections \ref{sec:about-the-proof}, \ref{sec:Dispersive-unknowns}. The structure of the nonlinearity as well as the nature of the dispersion links it to various fluid models e.g.\ the Euler-Coriolis system in $\R^3$.

Our main result quantifies the effect of the stratification-gravity coupling through the dispersion parameter $\kappa$ on the existence time of solutions to \eqref{eqn: perturbed BQ}. As a direct consequence we show that starting from initial data of only $L^2$-based regularity, the corresponding solution to \eqref{eqn: perturbed BQ} exists beyond the local well-posedness timescale. In what follows, we state the central theorem and discuss implications and existing literature.

\begin{theorem} \label{main-result-BQ}
    Let $n>3$. There exists $C_n>0$ such that if for an $\eps>0$ there holds
    \begin{align}\label{eq:intial-data-BQ}
    \norm{u_0}_{H^n}+\norm{\rho_0}_{H^n}= \eps,
    \end{align}
    then there exists 
    \begin{equation}
        T\geq C_n \kappa^{\frac{1}{3}} \varepsilon^{-\frac{4}{3}},
    \end{equation}
    and a unique solution $(u,\rho)\in C([0,T],H^n(\R^2))\times C([0,T],H^n(\R^2))$ to \eqref{eqn: perturbed BQ} with initial data $(u_0,\rho_0)$. In particular, the corresponding unique solution of \eqref{eqn: BQ} with initial data $(v_0,\rho_0)=(u_0,\rho_0-\kappa x_2)$ exists on the same timescale.
\end{theorem}

\smallskip
\begin{enumerate}[ align=left, leftmargin=0pt, 
          labelindent=0pt,listparindent=0pt, labelwidth=1em, itemindent=2em]
    \item(small data)
    One of the main implications of Theorem \ref{main-result-BQ} is that solutions to \eqref{eqn: perturbed BQ} with uniform stratification $\kappa=1$ starting from $\eps$-sized, \textit{$H^n$-regular} initial data with $\eps\ll 1$ exist on a timescale $\mathcal{O}(\eps^{-4/3})$.
    Previous attempts to stay in this regularity class have recovered, up to $\log$-improvements, the local well-posedness timescale $\mathcal{O}(\eps^{-1})$ \cite{wan2016, TakadaBQ2D}.  On the contrary, under \textit{additional localization} assumptions on the initial data, several authors have improved the local-well posedness timescale: refining the regularity obtained in the first result of \cite{Elgindi_2015}, \cite{wan2016} proved the $\mathcal{O}(\eps^{-4/3})$ lifespan for initial data in $H^{4^+}\cap \dot{B}^3_{1,1}$. Moreover, \cite{JW24} obtained the longest known existence timescale $\mathcal{O}(\eps^{-2})$ albeit under much higher Sobolev regularity and appropriate anisotropic localization assumptions, see \cite{Pusateri_2018, EC2022} for related models. \smallskip
    \item(strong stratification) Our result  also gives a qualitative answer to the  ``long-time solvability" question, i.e.\ given a time $T_*>0$ and fixed initial data, what is the minimal strength of stratification $\blue{\kappa(T_*)}$ such that a solution to \eqref{eqn: perturbed BQ} with $\kappa \geq \blue{\kappa(T_*)}$ exists up to time $\blue{T_*}$? Theorem \ref{main-result-BQ} shows that for  initial data of size $\eps\sim 1$ in $H^{3^+}$, a strength of dispersion $\kappa\gtrsim T_*^3$ guarantees existence up to the given time $T_*$. This improves over existing results in this regularity class, \cite{wan2016} obtaining the much rougher lower bound $\kappa\gtrsim T_*^4\exp(CT_*^{3/4})$.
    On the other hand, \cite{TakadaBQ2D} showed that the much stronger requirement $\kappa\gtrsim T_*^{7}\exp(CT_*)$ allows the solution to exist until time $T_*$ even for initial data in $H^{2^+}$, covering the entire local well-posedness class. \smallskip
       \item 
       In fact, the two viewpoints (1) ``small data" and (2) ``strong stratification" are equivalent as a consequence of the following scaling symmetry satisfied by \eqref{eqn: BQ}, \eqref{eqn: perturbed BQ}:

           \hspace{1em}\textit{Time-scaling.} If $(u,\rho)$ solve \eqref{eqn: BQ} ( resp.\ \eqref{eqn: perturbed BQ}) with strength of gravity $\kappa$ on the time interval $[0,T_\kappa]$ and initial data satisfying $\|(u_0,\rho_0)\|\leq \eps$, then the rescaled functions
    $(u_\kappa,\rho_\kappa)(t,x)=\kappa^{-1}(u,\rho)(\kappa^{-1}t,x)$ satisfy \eqref{eqn: BQ} (resp.\ \eqref{eqn: perturbed BQ}) with $\kappa=1$ on the time interval $[0,\kappa T_\kappa]$ with initial data $\|(u_{\kappa,0},\rho_{\kappa,0})\|\leq \kappa^{-1}\eps$. \smallskip
    \item (regularity) Our methods allow to take initial data $(u_0,\rho_0)\in \dot{H}^m\cap\dot{H}^n\supset H^n $, where $0 \leq m<1$ and $n>3$. This follows by refining the energy estimates (Proposition \ref{prop: energy estimates BQ}), Lemma \ref{product-lemma-final} and Proposition \ref{bootstrap-prop-final} to homogeneous Sobolev spaces in a straightforward way. In particular, our method allows for solutions with infinite $L^2$-norm, however, in Theorem \ref{main-result-BQ} we focus on the physically relevant solutions of finite kinetic energy.
\end{enumerate}

\bigskip

\par A closely related model is the dispersive SQG equation with the strength of dispersion $\kappa$:
\begin{equation}\label{SQG}
\begin{cases}
   \partial_t \theta +u \cdot \nabla \theta = \kappa \mathcal{R}_1\theta, \\
   u=\nabla^\perp(-\Delta)^{-1/2}\theta, \\
   \theta(0,x)=\theta_0(x),
\end{cases}\end{equation}
where $\theta:\R_+ \times \R^2\to \R$ is the temperature of the fluid and $\widehat{\mathcal{R}_1f}(\xi)\vcentcolon=-i{\xi_1}{\abs{\xi}^{-1}}\widehat{f}$ is the Riesz transform in the first coordinate. Its dispersive nature links it to the Boussinesq model above. The inviscid SQG equation ($\kappa=0$) has been proposed as a model for studying atmospheric turbulence, see e.g.\ \cite{P_Constantin_1994}.
The analogous result (with the same implications as above) for the dispersive SQG equation is summarized as follows:

\begin{theorem} \label{main-result-SQG}
    Let $n>3$. There exists $C_n>0$ such that if for an $\eps>0$ there holds
    \begin{align}\label{eq:intial-data-SQG}
    \norm{\theta_0}_{H^n}= \eps,
    \end{align}
   then there exists
    \begin{equation}
        T\geq C_n \kappa^{\frac{1}{3}} \varepsilon^{-\frac{4}{3}},
    \end{equation}
   and a unique solution $\theta\in C([0,T],H^n(\R^2))$ to \eqref{eqn: perturbed BQ} with initial data $\theta_0$.
   
\end{theorem}
\subsection{About the proof and plan of the article}\label{sec:about-the-proof}Investigating the role of stratification on the stability time goes back at least to works on rotating 3D fluids \cite{CDGG02,Math_Geophysics_Gallagher-Chemin}, \cite{Dut05}, \cite{KohLeeTakada-Str-Euler, KLT14-NS}, where rotation induces dispersion. We prove the main theorems via dispersive techniques akin to ones considered in the above works for rotating fluids and in \cite{Elgindi_2015, wan2016, TakBQ3D, TakadaBQ2D, charve-strichartz-2023, A-CFK-BQStrichartz, charve-strichartz-2025} in the context of stratified fluids. In the discussion below we highlight some key novelties, in particular \blue{the bootstrap framework, allowing for the use of inhomogeneous Strichartz estimates (similar to techniques used in the context of compressible Euler flows with rotation \cite{compr-euler-strichartz}).}

 The standard continuation criterion of the solution to \eqref{eqn: perturbed BQ} in $H^n$ involves controlling $\norm{\nabla(u,\rho)}_{L_t^1L_x^\infty}$, see e.g.\ Section \ref{sec:energy-estimates}. To that end, we want to exploit the dispersive nature of the problem. In Section \ref{sec:Dispersive-unknowns} we recall that the system \eqref{eqn: perturbed BQ} is equivalent to a nonlinear dispersive PDE \eqref{eqn: BQ on Z+, Z-} for scalar unknowns $Z_\pm$  \eqref{def:dispersive-unknowns} with dispersion relation $ \Lambda_\kappa(\xi)=\kappa{\xi_1}{\absxi^{-1}}$. The dispersive decay at the sharp  $(\kappa t)^{-1/2}$ rate has been shown in \cite{Elgindi_2015}, see Prop. \ref{prop:dispersive-decay}. 
 Hereafter we view $(u,\rho)$ in terms of $Z_\pm$ and consider the Duhamel formulation (see \eqref{Z+--in-prop}):
\begin{align}\label{intro-Duhamel}
    Z_\pm(t)=e^{\mp it\Lambda_\kappa}Z_{\pm,0}+\int_{0}^te^{\mp i(t-s)\Lambda_\kappa}\mathcal{NL}_\pm(Z_+,Z_-)ds,
\end{align}
where $\mathcal{NL}_\pm(Z_+,Z_-)$ encodes the nonlinearities $u\cdot \nabla u$, $u\cdot \nabla \rho$ in terms of $Z_\pm$. In this context, the central mechanism 
is a bootstrapping scheme detailed in Section \ref{sec:proof-main-thms} involving the energy norms $H^n$ and a bootstrap norm on the unknowns $Z_\pm$ discussed below. 

The key novelty of this paper is the choice of bootstrap framework (Proposition \ref{bootstrap-prop-final}), powered by inhomogeneous Strichartz estimates (Theorem \ref{Strichartz-Sobolev}) and a product lemma adjusted to this framework (Lemma \ref{product-lemma-final}). The main technical element of the proof is detailed in Proposition \ref{infty-norm-bound-final}, where we prove bounds on the bootstrap norm $\|(Z_+,Z_-)\|_{L_t^1\blue{ (\dot{B}_{\infty,1}^0\cap\dot{B}_{\infty,1}^1) }}$. 
Aside from  estimating the blow-up norm above, it allows to circumvent the unboundedness in $L^\infty$ of the Riesz transform appearing in $u$, $Z_\pm$, $\mathcal{NL}_\pm$. Indeed, employing the Besov norms resolves the issue of the Riesz transform as a result of handling objects that are localized in frequency, see Section \ref{sec:localizations}\blue{, Lemma \ref{product-lemma-final}}. However, this comes at the cost of an arbitrarily small $\delta$-derivative loss in $L^2$ due to summation, see e.g.\ \eqref{eqn:summation-L^2}, which resulted in missing $n=3$ in Theorem \ref{main-result-BQ}. Finally, this norm provides a good framework to employ Strichartz estimates discussed above. We note that while the use of Strichartz estimates on the linear term in \eqref{intro-Duhamel} does not provide the optimal linear timescale, see e.g.\ \eqref{eqn:linear-term-proof}, they allow us to keep the $L^2$-based regularity and match the timescale obtained from the nonlinear term.

\textbf{Plan of the Article. }In Section \ref{sec:setup}, we introduce the dispersive unknowns $Z_\pm$ for the Boussinesq system as well as the energy estimates available for the two problems. The dispersive decay and Strichartz estimates are presented in Section \ref{sec:Strichartz-estimates}. The detailed proof of the main theorems via bootstrap is conducted in Section \ref{sec:proof-main-thms}, where bounds on the bootstrap norm are shown in Section \ref{sec:bounds-on-boostrap-norm}.

\end{section}
\begin{section}{Setup}\label{sec:setup}
\subsection{Dispersive unknowns}\label{sec:Dispersive-unknowns}
It is well-known that the system \eqref{eqn: perturbed BQ} exhibits dispersion. For the purposes of this paper, we write the system in scalar unknowns such that the dispersive structure becomes explicit. We follow the approach from \cite[\textcolor{MidnightBlue}{Section 2.1}]{JW24} and explain the choice of unknowns, however we refer to \cite{JW24} for the complete proofs.
To this end, it is convenient to look at the system in vorticity form in order to have scalar unknowns $\omega, \rho:\R^+\times \R^2\to \R$ as
\begin{equation}\label{eqn:bq-vorticity}
\left\{
      \begin{aligned}
 &\partial_t\omega +u\cdot \nabla \omega=-\kappa\partial_{x_1}\rho,\\
 &\partial_t\rho+u\cdot \nabla\rho=\kappa\partial_{x_1}\Delta^{-1}\omega, \\
    &u=\nabla^\perp \Delta^{-1}\omega,
\end{aligned}
    \right.
\end{equation}
\blue{where $\omega = \nabla^\perp\cdot u$ is the vorticity.} The linearization in the frequency space reads:
\begin{equation}
\left\{
      \begin{aligned}
 &\partial_t\widehat{\omega}=-i\kappa\xi_1\widehat{\rho},\\
 &\partial_t\widehat{\rho}=-i\kappa \xi_1\abs{\nabla}^{-2}\omega.
\end{aligned}
    \right.
\end{equation}
We observe that $\F( \abs{\nabla}^{-1}\omega\pm \rho)$ diagonalize the linear system above with eigenvalues $\mp i\kappa \frac{\xi_1}{\absxi}$. Define the dispersive unknowns and the dispersion relation
\begin{equation}
   \begin{split}\label{def:dispersive-unknowns}
    Z_\pm&\vcentcolon=  \abs{\nabla}^{-1}\omega\pm \rho= \absnabla^{-1}\nabla^\perp\cdot u \pm \rho, \\
    \Lambda_\kappa(\xi)&\vcentcolon=\kappa \frac{\xi_1}{\absxi}. 
\end{split} 
\end{equation}

By abuse of notation, we denote by $\Lambda_\kappa$ also the pseudo-differential operator with symbol $\kappa \xi_1|\xi|^{-1}$, that is $\Lambda_\kappa=-\kappa\mathcal{R}_1$. Then linearly there holds $\partial_t Z_\pm= \mp i\Lambda_\kappa(\xi)Z_\pm.$
Note also that 
\begin{align}\label{eqn:unknowns-in-terms-of-Z}
    u=-\frac{1}{2}\nabla^{\perp}\abs{\nabla}^{-1}(Z_++Z_-), \qquad \rho=\frac{1}{2}(Z_+-Z_-), \qquad \omega=\frac{1}{2}\abs{\nabla}(Z_++Z_-),
\end{align}
and that the following energy balance holds by a direct computation
\begin{align}\label{eqn: energy conservation Z+ Z-}    
\norm{u}_{\dot{H}^k}^2+\norm{\rho}_{\dot{H}^k}^2=\frac{1}{2}\norm{Z_+}_{\dot{H}^k}^2+\frac{1}{2}\norm{Z_-}_{\dot{H}^k}^2,\quad k\in \blue{\N \cup \{0\}}.
    \end{align}
    The nonlinear system \eqref{eqn: perturbed BQ} can be written in terms of the dispersive unknowns $Z_\pm$ and it is this formulation that is best adapted to prove the main result.
    \begin{proposition}{\cite[\textcolor{MidnightBlue}{Proposition 2.1}]{JW24}}\label{prop: profiles BQ}
Let $(u,\rho)\in C([0,T], H^n(\R^2)\times H^n(\R^2))$ solve \eqref{eqn: perturbed BQ}.
Define the dispersive unknowns $Z_\pm$ and the dispersive operator $\Lambda_\kappa$ by \eqref{def:dispersive-unknowns}. 
Then $Z_\pm$ satisfy 
\begin{equation}\label{eqn: BQ on Z+, Z-}
    \blue{\begin{split}
         &\partial_tZ_\pm -\frac{1}{4}\absnabla^{-1}\mathrm{div}\Big(\nabla^\perp \absnabla^{-1}(Z_++Z_-)\cdot \absnabla(Z_++Z_-)\Big)\\
 &\hspace{3.2cm}\mp\frac{1}{4}\nabla^\perp\absnabla^{-1}(Z_++Z_-)\cdot\nabla(Z_+-Z_-)=\mp i\Lambda_\kappa Z_\pm.
    \end{split}}
\end{equation}
In particular, the Boussinesq system \eqref{eqn: perturbed BQ} on $(u,\rho)$ is equivalent to the system \eqref{eqn: BQ on Z+, Z-} on $Z_\pm$.
\end{proposition}

\subsection{Energy Estimates}\label{sec:energy-estimates}
The Boussinesq system and the SQG equation satisfy a priori energy estimates which follow via a standard proof (see \cite[\textcolor{MidnightBlue}{Propositions 4.2, 4.4}]{JW24},  \cite[\textcolor{MidnightBlue}{Lemma 3.1}]{Elgindi_2015}) and we record them in the following propositions.
\begin{proposition}\label{prop: energy estimates BQ}
    For any $n \geq 0$, let $(u,\rho)\in  C([0,T],H^n(\R^2))\times C([0,T],H^n(\R^2))$ solve \eqref{eqn: perturbed BQ} with initial data $(u_0,\rho_0)\in (H^n(\R^2))^2$ for $ T\geq 0$. Then for $0\leq t\leq T$ there holds
    \begin{align*}
        &\norm{u(t)}_{H^{n}}^2+\norm{\rho(t)}_{H^{n}}^2- \norm{u_0}_{H^{n}}^2-\norm{\rho_0}_{H^{n}}^2\lesssim \int_0^t\inftynorm{\nabla (u,\rho)(s)}(\norm{u(s)}_{H^n}^2+\norm{\rho(s)}_{H^{n}}^2)ds.
    \end{align*}
\end{proposition}

\begin{proposition}\label{prop: energy estimates SQG}
    Let $\theta\in C([0,T],H^n(\R^2))$ be a solution to \eqref{SQG} with initial data  $\theta_0\in H^n(\R^2)$ for $T\geq 0$ and $n\geq 0$. Then for $0\leq t\leq T$ there holds
  \begin{align*}
  \norm{\theta(t)}_{H^n}^2-\norm{\theta_0}_{H^n}^2&\lesssim \int_0^t\inftynorm{\nabla (u,\theta)(s)}\norm{\theta(s)}_{H^n}^2ds.
  \end{align*}
  \end{proposition}
  \subsection{Localizations}\label{sec:localizations}
In order to prove the main results, and in particular overcome the fact that the Riesz transform is not bounded in $L^\infty$, we will localize in frequency and use Besov norms for intermediate steps in the proof. Let $\psi\in C^{\infty}(\R, [0,1])$ be a even bump function with $\supp{\psi}\subset[-2,2]$ and $\psi|_{[-1,1]}\equiv 1$. Let $\varphi(x)\vcentcolon=\psi(x)-\psi(2x)$. We denote the standard  Littlewood-Paley projections $P_k$ of a function $f:\R^d\to \R$ by
\begin{align}
    \widehat{P_kf}(\xi)\vcentcolon=\varphi(2^{-k}|\xi|)\widehat{f}(\xi), \qquad k\in \Z.
\end{align}
These operators enjoy many useful properties that can be found in standard textbooks, see e.g.\ \cite[\textcolor{MidnightBlue}{Appendix A}]{taononlinear}. In particular, they can be used to define the homogeneous Besov norms used throughout this paper (see also \cite[\textcolor{MidnightBlue}{\S 2.7}]{Bahouri-Chemin-Danchin-Fourier}):
\begin{align}
    \|f\|_{\dot{B}_{p,q}^s}=\Big(\sum_{k\in \Z}(2^{sk}\norm{P_kf}_{L^p})^q\Big)^{\frac1q}. 
\end{align}
\end{section}
\begin{section}{Dispersive estimates}\label{sec:Strichartz-estimates}
The linear semigroup of the problem satisfies the following dispersive decay estimate proven in \cite[\textcolor{MidnightBlue}{Proposition 2.1}]{Elgindi_2015}.
\begin{proposition}\label{prop:dispersive-decay}
    Let $f\in C^{\infty}_c(\R^2)$ and $\Lambda_\kappa$ be given by \eqref{def:dispersive-unknowns}. Then there holds
    \begin{align*}
        \inftynorm{e^{\pm it\Lambda_\kappa}f}\lesssim (\kappa t)^{-1/2}\norm{f}_{\dot{B}_{1,1}^2}.
    \end{align*}
\end{proposition}
\begin{remark}
    In the following, we will use the dispersive estimate above for the dyadically localized pieces:
     \begin{equation}\label{local-piece-decay}
        \inftynorm{e^{it\Lambda_\kappa}P_k f}\leq C(\kappa t)^{-1/2}2^{2k}\norm{P_k f}_{L^1}.
    \end{equation}
\end{remark}
 A standard result using interpolation between the $L^1-L^\infty$ decay and the $L^2$ conservation gives the following decay rate of the semigroup in $L^p$ spaces.
    \begin{lemma}\label{lemma:lp-decay} For any $ 2\leq p\leq \infty$, there holds
       \begin{align*}
           \lpnorm{e^{it\Lambda_\kappa}P_kf}{p}\leq c_p (\kappa t)^{-\frac{1}{2}+\frac{1}{p}}2^{(2-\frac{4}{p})k}\lpnorm{P_kf}{p'}
       \end{align*} 
    \end{lemma}
  Finally, we can prove the main tool to achieve Theorems \ref{main-result-BQ}, \ref{main-result-SQG}. This is given by the following standard Strichartz estimates.
    \begin{theorem}\label{Strichartz-Sobolev}
        Let $f \in \mathcal{S}(\R^2) $, $F \in \mathcal{S}(\R^3)$ and $\Lambda_\kappa$ be the dispersion relation \eqref{def:dispersive-unknowns}. Moreover, let $(q,r)$ be a $\frac12$-admissible pair, that is
        \begin{align}
            \frac{1}{q}+\frac{1}{2r}= \frac{1}{4}, \qquad 4\leq q\leq \infty, \; 2\leq r\leq \infty.
        \end{align}
        Then, from \eqref{local-piece-decay}, there holds
        \begin{enumerate}
            \item \label{Strichartz-1}$\norm{e^{it\Lambda_\kappa}P_k f}_{L^q_tL^r_x}\lesssim_q \kappa^{-\frac{1}{q}} 2^{\frac{4k}{q}} \ltwonorm{P_k f}$;
            \item \label{Strichartz-2}$\norm{\int_\R e^{-is\Lambda_\kappa}P_k F(s,x)ds}_{L^2_x} \lesssim_q \kappa^{-\frac{1}{q}} 2^{\frac{4}{q}k}\norm{P_kF}_{L_t^{q'}L_x^{r'}}$, where $\frac{1}{q}+\frac{1}{q'}=\frac{1}{r}+\frac{1}{r'}=1$;
            \item \label{Strichartz-3} for any $(q_i,r_i)$ $\frac12-$admissible pairs, $i=1,2$ there holds
            \begin{align*}
                \norm{\int_0^t e^{i(t-s)\Lambda_\kappa}P_k F(s,x)ds}_{L_t^{q_1}L_x^{r_1}} \lesssim_{q_1,q_2} \kappa^{-\frac{1}{q_1} - \frac{1}{q_2}} 2^{\frac{4k}{q_1} + \frac{4k}{q_2}} \norm{P_k F}_{L_t^{q_2'}L_x^{r_2'}},  
            \end{align*}
          where $\frac{1}{q_i}+\frac{1}{q_i'}=\frac{1}{r_i}+\frac{1}{r_i'}=1$.
        \end{enumerate}
    \end{theorem}
    \begin{proof}
     The proof begins by establishing the second statement. We compute using Lemma \ref{lemma:lp-decay} and the Hardy-Littlewood-Sobolev inequality:
      \begin{align*}
          \ltwonorm{\int_\R e^{-is\Lambda_\kappa}P_kF(s,\cdot)ds}^2
          &= \int_\R\int_\R \langle e^{i(t-s)\Lambda_\kappa}P_kF(s,\cdot),P_kF(t,\cdot)\rangle_{L^2_x}dsdt\\
          &\leq \int_\R\int_\R \norm{e^{i(t-s)\Lambda_\kappa}P_kF(s,\cdot)}_{L_x^r} ds \norm{P_kF(t,\cdot)}_{L_x^{r'}}dt\\
          &\lesssim \int_\R\int_\R (\kappa(t-s))^{-\frac{1}{2}+\frac{1}{r}}2^{(2-\frac{4}{r})k}\norm{P_kF(s,\cdot)}_{L_x^{r'}}ds \norm{P_kF(t,\cdot)}_{L_x^{r'}}dt\\
          &\lesssim \kappa^{-\frac{1}{2}+\frac{1}{r}} 2^{(2-\frac{4}{r})k}\norm{|t|^{-\frac{1}{2}+\frac{1}{r}}*\norm{P_kF(t,\cdot)}_{L_x^{r'}}}_{L_t^q}\norm{P_kF(t,\cdot)}_{L_t^{q'}L_x^{r'}}\\
          &\lesssim \kappa^{-\frac{2}{q}} 2^{\frac{8}{q}k} \norm{P_kF(t,\cdot)}_{L_t^{q'}L_x^{r'}}^2.
      \end{align*}
      \par The first statement follows from the second by duality. Indeed, for $F\in L_t^{q'}L_x^{r'}$ there holds
      \begin{align*}
          \langle e^{it\Lambda_\kappa} P_kf, F\rangle_{ L_t^{q}L_x^{r}}&=\int_t\int_x  e^{it\Lambda_\kappa} P_kf(x)\overline{F(t,x)}dx dt\\
          &=\langle  P_kf,\int_\R e^{-it\Lambda_\kappa}F(t,\cdot)dt\rangle_{L_x^2}\\
          &\leq \ltwonorm{P_k f}\ltwonorm{\int_\R e^{-it\Lambda_\kappa}\widetilde{P_k}F(t,\cdot)dt}\\
          &\lesssim \kappa^{-\frac{1}{q}} 2^{\frac{4}{q}k}\ltwonorm{P_kf} \norm{F}_{L_t^{q'}L_x^{r'}},
      \end{align*}
      where $\widetilde{P_k}$ is the Littlewood-Paley operator associated to a flattened bump $\tilde{\varphi}$ with $\tilde{\varphi}\equiv1$ on $\mathrm{supp}(\varphi)$.
      Finally, by combining \eqref{Strichartz-1} and \eqref{Strichartz-2},
      \begin{align*}
          \norm{\int_\R e^{i(t-s)\Lambda_\kappa}P_kF(s,x)ds}_{L_t^{q_1}L_x^{r_1}}&=\norm{e^{it\Lambda_\kappa}\int_\R e^{-is\Lambda_\kappa}P_kF(s,x)ds}_{L_t^{q_1}L_x^{r_1}}\\
          &\leq c\kappa^{-\frac{1}{q_1}}2^{\frac{4}{q_1}k}\ltwonorm{\int_\R e^{-is\Lambda_\kappa}P_kF(s,x)ds}\\
          &\leq c\kappa^{-\frac{1}{q_1}}2^{\frac{4}{q_1}k}\kappa^{-\frac{1}{q_2}}2^{\frac{4}{q_2}k}\norm{P_kF}_{L_t^{q_2'}L_x^{r_2'}},
      \end{align*}
      from which we can obtain \eqref{Strichartz-3} using Christ-Kiselev lemma \cite{Christ-Kiselev} as in \cite{taononlinear}.
 
    \end{proof}

\end{section}

\section{Proof of Theorems \ref{main-result-BQ},\ref{main-result-SQG}}\label{sec:proof-main-thms}

\blue{The proof follows via a bootstrap argument presented in Proposition \ref{bootstrap-prop-final}. For notational convenience, we define the space in which we conduct the bootstrap argument by $B := \dot{B}^0_{\infty,1} \cap \dot{B}^1_{\infty,1}$. In particular, the respective norm $\norm{\cdot}_B$ for a function $f$ is given by
\begin{equation}\label{eqn:B-norm}
    \norm{f}_B \vcentcolon = \norm{f}_{\dot{B}^0_{\infty,1}} + \norm{f}_{\dot{B}^1_{\infty,1}}.
\end{equation}}
  \begin{proof}[Proof of Theorems \ref{main-result-BQ}, \ref{main-result-SQG}]\label{proof:main-result-BQ} We recall that the Boussinesq system \eqref{eqn: perturbed BQ} is equivalent to the respective system  \eqref{eqn: BQ on Z+, Z-} for the dispersive unknowns $Z_\pm$ \blue{with the initial data $Z_{\pm,0}=Z_\pm(0)$}. Due to the energy balance \eqref{eqn: energy conservation Z+ Z-}, $Z_{\pm,0}$ satisfies
  
\begin{align}
    \label{eqn:initial-data-BQ-Z-pm}
    \|Z_{\pm,0}\|_{H^n}\leq \sqrt{2}\eps.
\end{align}
    In this framework, the theorem follows from the local well-posedness theory for the Boussinesq equation and a continuity argument provided by Proposition \ref{bootstrap-prop-final} relying on the Strichartz estimates from Section \ref{sec:Strichartz-estimates}. The bootstrap argument is carried on the energy norms $\|\cdot\|_{L_t^\infty H^n_x}$ and the Besov norms \blue{$\|\cdot\|_{L_t^1 B_x}$}. The choice of this $L^\infty$-based norm in space is primarily  motivated by the blow-up criterion and is discussed in the introduction Section \ref{sec:about-the-proof} and in the proof of Proposition \ref{bootstrap-prop-final}. We note here that the bootstrap norm is locally bounded for $t>0$ within the local well-posedness timescale and for $\delta \in (0,n-2]$:
\begin{align*}
    \blue{\norm{Z_\pm(t)}_{B}} =\sum_{k\in \Z}(2^k+1)\|P_kZ_\pm(t)\|_{L^\infty}\lesssim \sum_{k\in \Z}2^k(2^k+1)\|P_kZ_\pm(t)\|_{L^2}\lesssim \|Z_\pm(t)\|_{H^{2+\delta}},
\end{align*}
since the initial data is in $H^n$ for $n>3$.
Theorem \ref{main-result-SQG} for the SQG equation follows similarly from the local well-posedness and the respective statements for a solution $\theta$ from Proposition \ref{bootstrap-prop-final}.
\end{proof}

\subsection{Bootstrap argument}
\begin{proposition}\label{bootstrap-prop-final}
Let $n>3$ and \blue{$4 \leq q\leq \infty$}. Then there exists $C_{q,n}>0$ such that if $T \leq C_{q,n} \kappa^{\frac{1}{q-1}} \varepsilon^{-\frac{q}{q-1}}$, and $Z_\pm \in C([0,T], H^n(\R^2))$ resp. $\theta\in C([0,T], H^n(\R^2))$ solve \eqref{eqn: BQ on Z+, Z-} resp. \eqref{SQG} with initial data satisfying \eqref{eqn:initial-data-BQ-Z-pm} resp. \eqref{eq:intial-data-SQG} and moreover if for any $\tau \in [0,T]$ there holds
\begin{equation}
\begin{gathered}\label{bootstrap-assumption-second}
    \norm{(Z_+,Z_-)}_{L_t^\infty([0,\tau], H_x^n)}\leq 100\eps, \quad \text{resp.} \quad \norm{\theta}_{L_t^\infty([0,\tau], H_x^n)}\leq 100 \eps, \\
    \norm{(Z_+,Z_-)}_{L_t^1([0,\tau], \blue{B_x})} \leq 2, \quad \text{resp.} \quad \norm{\theta}_{L_t^1([0,\tau], \blue{B_x})} \leq 2,
\end{gathered}
\end{equation}
    then in fact the \blue{improved bounds} hold:
\begin{equation}\label{eqn:bootstrap-improved}
\begin{gathered}
    \norm{(Z_+,Z_-)}_{L_t^\infty([0,\tau], H_x^n)}\leq 10\eps, \quad \text{resp.} \quad \norm{\theta}_{L_t^\infty([0,\tau], H_x^n)}\leq 10 \eps, \\
    \norm{(Z_+,Z_-)}_{L_t^1([0,\tau], \blue{B_x})} \leq 1, \quad \text{resp.} \quad \norm{\theta}_{L_t^1([0,\tau], \blue{B_x})} \leq 1.
\end{gathered}
\end{equation}
\end{proposition}

\begin{proof}
We start by proving the statement for the SQG equation \eqref{SQG} to clearly illustrate the main argument. By the energy estimates we have the blow-up criterion
\begin{equation}
    \norm{\theta(\tau)}_{H^n}^2 \leq \norm{\theta_0}_{H^n}^2\exp\{ {K_n}\norm{\nabla (u, \theta)}_{L_t^1([0,\tau], L_x^\infty)}\}. 
\end{equation}
From Proposition \ref{infty-norm-bound-final} with $\delta=n-3$, we have
\begin{align}
    \norm{\nabla (u,\theta)}_{L_t^1([0,\tau], L_x^\infty)}& \leq K \sum_{k\in \Z} 2^k( \|P_k\nabla^\perp(-\Delta)^{-1/2}\theta\|_{L_t^1([0,\tau], L_x^\infty)} + \|P_k\theta\|_{L_t^1([0,\tau], L_x^\infty)})\\
    &\leq K\norm{\theta}_{L_t^1([0,\tau], \blue{B_x})} \\
    & \leq K[ K_{q,n} \tau^{1-\frac{1}{q}} \kappa^{-\frac{1}{q}} \norm{\theta_0}_{H^n}  + K_{q,n} \tau^{1-\frac{1}{q}} \kappa^{-\frac{1}{q}}\norm{\theta}_{H^n} \norm{\theta}_{L_t^1([0,\tau],\blue{B_x})}].
\end{align}
for a universal constant $K_{q,n}>0$. Hence, if $\tau \leq T$, together with the bootstrap assumption \eqref{bootstrap-assumption-second} we obtain the improved bounds \eqref{eqn:bootstrap-improved} by choosing $C_{q,n}>0$ such that
\begin{equation}\label{eqn:choice-of-constant-bootstrap}
\left\{
      \begin{aligned}
 &\norm{\theta}_{L_t^1([0,\tau], \blue{B_x})} \leq 300 K_{q,n} C_{q,n}^{\frac{q-1}{q}} \leq 1,\\
   &\exp\{300{K_n}K K_{q,n} C_{q,n}^{\frac{q-1}{q}} \}\leq 100.
\end{aligned}
    \right.
\end{equation}
 The Boussinesq system is treated analogously. By the energy estimates for $(u,\rho)$ provided by Proposition \ref{prop: energy estimates BQ} and the energy balance \eqref{eqn: energy conservation Z+ Z-} there holds:
\begin{align*}
    \|Z_+\|^2_{H^n}+\|Z_-\|^2_{H^n}\leq (\|Z_{+,0}\|^2_{H^n}+\|Z_{-,0}\|^2_{H^n})\exp\{K_n( \|\nabla u\|_{L_t^1([0,\tau], L_x^\infty)}+\|\nabla \rho\|_{L_t^1([0,\tau], L_x^\infty)})\}.
\end{align*}
The terms on the right-hand side are controlled by the bootstrap norm as follows:
 \begin{align*}
        \norm{\nabla (u, \rho)}_{L_t^1L_x^\infty}&= \frac12\|{\nabla \nabla^\perp\absnabla^{-1}(Z_++Z_-)}\|_{L_t^1L_x^\infty}+\frac12\norm{\nabla(Z_+-Z_-)}_{L_t^1L_x^\infty}\\
        &\leq K \sum_k 2^{k}\|P_k\nabla^\perp\absnabla^{-1}(Z_++Z_-)\|_{L_t^1L_x^\infty}+  \sum_k 2^{k}\|P_k(Z_+-Z_-)\|_{L_t^1L_x^\infty}\\
        &\leq K \|(Z_+,Z_-)\|_{L_t^1\blue{B_x}}.
    \end{align*}
    Finally, Proposition \ref{infty-norm-bound-final} provides a universal constant $K_{q,n}>0$ such that
    \begin{align*}
         & \|(Z_+,Z_-)\|_{L_t^1([0,\tau], \blue{B_x})}\\
         &\qquad\qquad\leq K_{q,n} \tau^{1-\frac{1}{q}} \kappa^{-\frac{1}{q}} \left(\norm{(Z_{+,0},Z_{-,0})}_{H^n}+ \|(Z_+,Z_-)\|_{L_t^1([0,\tau], \blue{B_x})}\norm{(Z_{+},Z_-)}_{H^{n}}\right).
    \end{align*}
    The improved bounds follow with the same choice of constant $C_{q,n}$ as in \eqref{eqn:choice-of-constant-bootstrap}.
\end{proof}

\subsection{Bounds on the bootstrap norm}\label{sec:bounds-on-boostrap-norm}
Finally, we prove the main bound on the bootstrap norm $\|\cdot\|_{L_t^1\blue{B_x}}$ defined in \eqref{eqn:B-norm} in Proposition \ref{infty-norm-bound-final} below. In order to control nonlinear terms in \eqref{eqn:SQG-Duhamel}, \eqref{Z+--in-prop}, we will use the following product lemma, where we see the bootstrap norm of Proposition \ref{bootstrap-prop-final} appears.
\begin{lemma}
For any $m \geq 0$ and $f \in \dot{B}^{1}_{\infty,1}(\mathbb{R}^d) \cap H^{m+1}(\mathbb{R}^d)$, $g \in \dot{B}^{0}_{\infty,1}(\mathbb{R}^d) \cap H^{m}(\mathbb{R}^d)$, if $u = \nabla^\perp (-\Delta)^{-1/2}g$, the following holds:
\begin{equation}\label{product-lemma-final}
    \norm{u\cdot\nabla f}_{H^m} + \norm{u \cdot (|\nabla|f)}_{H^m} \lesssim \norm{g}_{H^m}\norm{ f}_{\dot{B}^1_{\infty,1}} + \norm{g}_{\dot{B}^{0}_{\infty,1}} \norm{f}_{H^{m+1}}.
\end{equation}
\end{lemma}
\begin{proof}
\blue{The result is direct when $m=0$, hence we only consider $m>0$. Using the standard paraproduct estimates, \cite[\textcolor{MidnightBlue}{Corollary 2.86}]{Bahouri-Chemin-Danchin-Fourier} for example, we have
\begin{align}
    \|u\cdot \nabla f\|_{H^m}\lesssim_m \|u\|_{L^\infty}\|\nabla f\|_{H^m}+  \|\nabla f\|_{L^\infty}\| u\|_{H^m}.
\end{align}
To deal with the unboundedness of the Riesz transform in $L^\infty$, we use the embedding $\dot{B}^0_{\infty,1} \hookrightarrow L^\infty$ and the fact that this operator is bounded in the smaller space (see e.g.\ \cite[\textcolor{MidnightBlue}{Theorem 1.3}]{Y20-BesovSpaces}). This results in
\begin{align}
    \|u\cdot \nabla f\|_{H^m} & \lesssim \sum_{k}\|P_ku\|_{L^\infty}\|\nabla f\|_{H^m}+ \sum_{k} 2^k\|P_k f\|_{L^\infty}\| u\|_{H^m} \\
    & \lesssim \norm{g}_{\dot{B}^{0}_{\infty,1}} \norm{f}_{H^{m+1}} + \norm{f}_{\dot{B}^1_{\infty,1}}\norm{g}_{H^m}.
\end{align}
The same estimates follow for $u\cdot |\nabla|f$ and this yields the claim of the lemma.}

\end{proof}

\begin{proposition}[Bounds on the bootstrap norm]\label{infty-norm-bound-final} Let \blue{$4 \leq q \leq \infty$} and $\delta > 0$.
\begin{enumerate}
    \item Let $Z_\pm$ be the dispersive unknowns defined in \eqref{def:dispersive-unknowns} that solve \eqref{eqn: BQ on Z+, Z-} with the corresponding initial data $Z_{\pm,0}\in H^{3+\delta}(\R^2)$. Then, there holds
    \begin{align*}
       & \|(Z_+,Z_-)\|_{L_t^1([0,\tau], \blue{B_x})}\\
       &\qquad\leq K_{q,\delta} \tau^{1-\frac{1}{q}} \kappa^{-\frac{1}{q}} \left(\norm{(Z_{+,0},Z_{-,0})}_{H^{2+\delta}_{\blue{x}}}+ \|(Z_+,Z_-)\|_{L_t^1([0,\tau], \blue{B_x})}\norm{(Z_{+},Z_-)}_{L_t^\infty([0,\tau], H^{3+\delta}_{\blue{x}})}\right).
    \end{align*}

\item Let $\theta$ be a solution to \eqref{SQG} with $\theta_0 \in H^{3+\delta}(\mathbb{R}^2)$. Then, there holds
\begin{align*}
    \norm{\theta}_{L_t^1([0,\tau], \blue{B_x})}\leq K_{q,\delta} \tau^{1-\frac{1}{q}} \kappa^{-\frac{1}{q}} \left(\norm{\theta_0}_{H^{2+\delta}_{\blue{x}}}+  \norm{\theta}_{L_t^1([0,\tau], \blue{B_x})} \norm{\theta}_{L_t^\infty([0,\tau], H^{3+\delta}_{\blue{x}})}\right).
\end{align*}
\end{enumerate}
\end{proposition}

\begin{proof}
    We first prove the proposition for the SQG equation, (2), as the notation is more clear and compact. We recall the Duhamel formula:
    \begin{align}\label{eqn:SQG-Duhamel}
        \blue{\theta(t)=e^{-it\Lambda_\kappa}\theta_0 - \int_0^t e^{-i(t-s)\Lambda_\kappa}(u\cdot\nabla\theta(s)) ds.}
    \end{align}
We estimate the two terms on the right-hand side separately. We start with the linear term and use the Bernstein inequality and Theorem \ref{Strichartz-Sobolev}\eqref{Strichartz-1} with $(q,r)$ an $\frac12-$admissible pair:
\begin{equation}\label{eqn:linear-term-proof}
   \begin{split}
     \norm{e^{-it\Lambda_\kappa}\theta_0}_{L_t^1([0,\tau], \blue{B_x})}&\lesssim \tau^{1-\frac1q} \sum_k(2^k+1)\|P_k e^{-it\Lambda_\kappa}\theta_0\|_{L_t^qL^\infty_x}  \\
     &\lesssim \tau^{1-\frac1q}\sum_k (2^{k}+1)2^{\frac{2}{r}k}\|{P_ke^{-it\Lambda_\kappa}\theta_0}\|_{L_t^qL_x^r}\\
        &\lesssim \tau^{1-\frac1q}\kappa^{-\frac1q}\sum_k (2^{k}+1) 2^{(\frac{4}{q}+\frac{2}{r})k}\|P_k\theta_0\|_{L^2}\\
        &\lesssim \tau^{1-\frac1q}\kappa^{-\frac1q}\|\theta_0\|_{H^{2+\delta}}.
\end{split} 
\end{equation}

   The last inequality follows since for any $l > 0$ and $\delta > 0$, there holds
\begin{equation}\label{eqn:summation-L^2}
    \sum_k 2^{lk} \norm{P_k f}_{L^2} \leq \sum_{k \leq 0} 2^{lk} \norm{P_k f}_{L^2} + \sum_{k > 0} 2^{lk}  2^{-(l+\delta)k} \norm{P_k f}_{H^{l+\delta}} \lesssim \norm{f}_{H^{l+\delta}}.
\end{equation} 
   To bound the nonlinear term in \eqref{eqn:SQG-Duhamel}, we proceed as above and then apply Theorem \ref{Strichartz-Sobolev}\eqref{Strichartz-3} with $(q_2,r_2)=(\infty,2)$ as well as Lemma \ref{product-lemma-final} to estimate the product in the last line:
   \begin{align*}
         \| \int_0^t e^{-i(t-s)\Lambda_\kappa}u\cdot \nabla\theta(s)ds&\|_{L_t^1([0,\tau], \blue{B_x})}\\
       &\lesssim \tau^{1-\frac1q}\sum_k (2^k+1)2^{\frac2r k}\|P_k\int_0^t e^{-i(t-s)\Lambda_\kappa}u\cdot\nabla \theta(s)ds\|_{L_t^q([0,\tau], L^r_x)}\\
       &\lesssim \tau^{1-\frac1q}\kappa^{-\frac1q}\sum_{k} (2^k+1)2^{(\frac{4}{q}+\frac2r)k}\|P_k(u\cdot\nabla\theta)\|_{L_t^1L_x^2}\\
       &\lesssim \tau^{1-\frac1q}\kappa^{-\frac1q}\norm{u\cdot\nabla\theta}_{L_t^1H_x^{2+\delta}}\\
    & \lesssim \tau^{1-\frac1q}\kappa^{-\frac1q}\|\theta\|_{L_t^1\blue{B_x}}\|\theta\|_{L_t^\infty H^{3+\delta}_x}.
   \end{align*}
   
  The statement for the Boussinesq system follows in complete analogy using the dispersive variables $Z_\pm$.  Recall that $Z_\pm $ solves \eqref{eqn: BQ on Z+, Z-} and so the Duhamel formula reads
    \begin{align}\label{Z+--in-prop}
        \blue{Z_{\pm}(t)=e^{\mp it\Lambda_\kappa}Z_{\pm,0} + \int_0^t e^{\mp i(t-s)\Lambda_\kappa}\mathcal{NL}_\pm(Z_+,Z_-)ds,}
    \end{align}
    where by \eqref{eqn: BQ on Z+, Z-}
    \begin{align}\label{eqn:nonlinearity-Z_pm}
        \mathcal{NL}_\pm(Z_+,Z_-)&\vcentcolon= \frac{1}{4}\absnabla^{-1}\mathrm{div}\Big(\nabla^\perp \absnabla^{-1}(Z_++Z_-)\cdot \absnabla(Z_++Z_-)\Big)\\
        &\qquad\pm\frac{1}{4}\nabla^\perp\absnabla^{-1}(Z_++Z_-)\cdot\nabla(Z_+-Z_-).
    \end{align} 
    Therefore, in order to bound the dispersive unknowns $Z_\pm$ in the bootstrap norm, we need to control the linear and non-linear contribution from \eqref{Z+--in-prop} for each $Z_\pm$. The linear contribution follows using Theorem \ref{Strichartz-Sobolev}\eqref{Strichartz-1} and the Bernstein inequality:
    \begin{align*}
        \|e^{\mp i t\Lambda_\kappa}Z_{\pm,0}\|_{L_t^1\blue{B_x}}\lesssim \tau^{1-\frac1q}\kappa^{-\frac1q}\sum_{k}(2^{k}+1)2^{(\frac{2}{r}+\frac4q)k}\norm{P_kZ_{\pm,0}}_{L^2_x}\lesssim \tau^{1-\frac1q}\kappa^{-\frac1q}\norm{Z_{\pm,0}}_{H_x^{2+\delta}}.
    \end{align*}
    The bound for nonlinear term follows as in the SQG case using Theorem \ref{Strichartz-Sobolev}\eqref{Strichartz-3} with $(q_2,r_2)=(\infty,2)$ and Lemma \ref{product-lemma-final} (through the identification $(-\Delta)^{-1/2}\simeq \absnabla^{-1}$ in the Fourier space) in the last inequality:
    \begin{align*}
          \|\int_0^te^{\mp i (t-s)\Lambda_\kappa}\mathcal{NL}_\pm ds\|_{L_t^1([0,\tau], \blue{B_x})}&\lesssim \tau^{1-\frac1q}\kappa^{-\frac1q}\sum_{k}(2^{k}+1)2^{(\frac{2}{r}+\frac4q)k}\norm{P_k\mathcal{NL}_\pm}_{L_t^1L^2_x}\\
          &\lesssim \tau^{1-\frac1q}\kappa^{-\frac1q} \sum_{\mu,\nu\in\{+,-\}} \left[ \|{\nabla^\perp\absnabla^{-1}Z_\mu\cdot \absnabla Z_\nu}\|_{L_t^1H^{2+\delta}_x}  \right. \\
          & \hspace{35mm} \left. +\|{\nabla^\perp\absnabla^{-1}Z_\mu\cdot \nabla Z_\nu}\|_{L_t^1H^{2+\delta}_x} \right]\\
          &\lesssim  \tau^{1-\frac1q}\kappa^{-\frac1q} \sum_{\mu,\nu\in\{+,-\}} \|Z_\mu\|_{L_t^1\blue{B_x}}\|Z_\nu\|_{L_t^\infty H^{3+\delta}_x}.
    \end{align*}
    Altogether we obtain the claim 
    \begin{align*}
        \norm{(Z_+,Z_-)}_{L_t^1([0,\tau], \blue{B_x})}\lesssim \tau^{1-\frac1q}\kappa^{-\frac1q}(\|(Z_{+,0},Z_{-,0})\|_{H^{2+\delta}_x}+ \|(Z_{+},Z_{-})\|_{L_t^1([0,\tau], \blue{B_x})}\|(Z_{+},Z_{-})\|_{H^{3+\delta}_x}).
    \end{align*}
    
\end{proof}

\section*{Acknowledgments}
The authors would like to thank Klaus Widmayer and Beno\^it Pausader for many insightful conversations.

The first author gratefully acknowledges the support of the SNSF through grant PCEFP2\_203059. The second author was partially supported by NSF grant DMS-2154162, DMS-2452275, and Brown University Graduate School Travel Awards, and appreciates the hospitality of the University of Zurich during his visit, where this work was initiated.

\addtocontents{toc}{\protect\setcounter{tocdepth}{1}}
\bibliographystyle{alpha}
\bibliography{FINAL-bib}

\begin{thebibliography}{KPTW25}

\bibitem[ACFK24]{A-CFK-BQStrichartz}
V.~Angulo-Castillo, L.~C.~F. Ferreira, and L.~Kosloff.
\newblock Long-time solvability for the 2d inviscid {Boussinesq} equations with
  borderline regularity and dispersive effects.
\newblock {\em Asymptotic Anal.}, 137(1-2):53--84, 2024.

\bibitem[BCD11]{Bahouri-Chemin-Danchin-Fourier}
Hajer Bahouri, Jean-Yves Chemin, and Rapha{\"e}l Danchin.
\newblock {\em Fourier analysis and nonlinear partial differential equations},
  volume 343 of {\em Grundlehren Math. Wiss.}
\newblock Berlin: Springer, 2011.

\bibitem[BHI24]{bianchini2024strongillposednesslinfty2D}
Roberta Bianchini, Lars~Eric Hientzsch, and Felice Iandoli.
\newblock Strong ill-{Posedness} in {{\(L^\infty\)}} of the 2d {Boussinesq}
  equations in vorticity form and application to the 3d axisymmetric {Euler}
  equations.
\newblock {\em SIAM J. Math. Anal.}, 56(5):5915--5968, 2024.

\bibitem[CDGG02]{CDGG02}
J.-Y. Chemin, B.~Desjardins, I.~Gallagher, and E.~Grenier.
\newblock Anisotropy and dispersion in rotating fluids.
\newblock In {\em Nonlinear partial differential equations and their
  applications. {C}oll\`ege de {F}rance {S}eminar, {V}ol. {XIV} ({P}aris,
  1997/1998)}, volume~31 of {\em Stud. Math. Appl.}, pages 171--192.
  North-Holland, Amsterdam, 2002.

\bibitem[CDGG06]{Math_Geophysics_Gallagher-Chemin}
Jean-Yves Chemin, Benoit Desjardins, Isabelle Gallagher, and Emmanuel Grenier.
\newblock {\em Mathematical geophysics. {An} introduction to rotating fluids
  and the {Navier}-{Stokes} equations.}, volume~32 of {\em Oxf. Lect. Ser.
  Math. Appl.}
\newblock Oxford: Clarendon Press, 2006.

\bibitem[CH23]{CH22}
Jiajie Chen and Thomas~Y. Hou.
\newblock Stable nearly self-similar blowup of the {2D Boussinesq} and {3D
  Euler} equations with smooth data {I}: {A}nalysis.
\newblock {\em arXiv e-prints}, arXiv:2210.07191, 2023.

\bibitem[CH25]{CH24}
Jiajie Chen and Thomas~Y. Hou.
\newblock Stable nearly self-similar blowup of the {2D Boussinesq} and {3D
  Euler} equations with smooth data {II}: {Rigorous Numerics}.
\newblock {\em Multiscale Model. Simul.}, 23(1):25--130, 2025.

\bibitem[Cha23]{charve-strichartz-2023}
Fr{\'e}d{\'e}ric Charve.
\newblock Sharper dispersive estimates and asymptotics for a {Boussinesq}-type
  system with larger ill-prepared initial data.
\newblock {\em Asymptotic Anal.}, 131(3-4):443--470, 2023.

\bibitem[Cha25]{charve-strichartz-2025}
Fr{\'e}d{\'e}ric Charve.
\newblock Hidden asymptotics for the weak solutions of the strongly stratified
  {Boussinesq} system without rotation.
\newblock {\em J. Math. Pures Appl. (9)}, 202:59, 2025.
\newblock Id/No 103750.

\bibitem[CK01]{Christ-Kiselev}
Michael Christ and Alexander Kiselev.
\newblock Maximal functions associated to filtrations.
\newblock {\em J. Funct. Anal.}, 179(2):409--425, 2001.

\bibitem[CKN99]{BQ_LWP_Holder}
Dongho Chae, Sung-Ki Kim, and Hee-Seok Nam.
\newblock Local existence and blow-up criterion of {H{\"o}lder} continuous
  solutions of the {Boussinesq} equations.
\newblock {\em Nagoya Math. J.}, 155:55--80, 1999.

\bibitem[CMT94]{P_Constantin_1994}
Peter Constantin, Andrew~J. Majda, and Esteban Tabak.
\newblock Formation of strong fronts in the 2-{D} quasigeostrophic thermal
  active scalar.
\newblock {\em Nonlinearity}, 7(6):1495--1553, 1994.

\bibitem[CN97]{LWP_Chae}
Dongho Chae and Hee-Seok Nam.
\newblock Local existence and blow-up criterion for the {Boussinesq} equations.
\newblock {\em Proc. R. Soc. Edinb., Sect. A, Math.}, 127(5):935--946, 1997.

\bibitem[Dut05]{Dut05}
Alexandre Dutrifoy.
\newblock Examples of dispersive effects in non-viscous rotating fluids.
\newblock {\em J. Math. Pures Appl. (9)}, 84(3):331--356, 2005.

\bibitem[EJ20]{Elgindi-Jeong-SingFormBQ}
Tarek~M. Elgindi and In-Jee Jeong.
\newblock Finite-time singularity formation for strong solutions to the
  {Boussinesq} system.
\newblock {\em Ann. PDE}, 6(1):50, 2020.
\newblock Id/No 5.

\bibitem[EW15]{Elgindi_2015}
Tarek~M. Elgindi and Klaus Widmayer.
\newblock Sharp decay estimates for an anisotropic linear semigroup and
  applications to the surface quasi-geostrophic and inviscid {Boussinesq}
  systems.
\newblock {\em SIAM J. Math. Anal.}, 47(6):4672--4684, 2015.

\bibitem[Gal20]{gallay2019stabilityvorticesidealfluids}
Thierry Gallay.
\newblock Stability of vortices in ideal fluids: the legacy of {Kelvin} and
  {Rayleigh}.
\newblock In {\em Hyperbolic problems: theory, numerics, applications.
  Proceedings of the 17th international conference, HYP2018, Pennsylvania State
  University, University Park, PA, USA, June 25--29, 2018}, pages 42--59. 2020.

\bibitem[GPW23]{EC2022}
Yan Guo, Benoit Pausader, and Klaus Widmayer.
\newblock Global axisymmetric {Euler} flows with rotation.
\newblock {\em Invent. Math.}, 231(1):169--262, 2023.

\bibitem[JW24]{JW24}
Catalina Jurja and Klaus Widmayer.
\newblock Long-time stability of a stably stratified rest state in the inviscid
  2d {Boussinesq} equation.
\newblock {\em arXiv e-prints}, arXiv:2408.15154, 2024.

\bibitem[KLT14a]{KLT14-NS}
Youngwoo Koh, Sanghyuk Lee, and Ryo Takada.
\newblock Dispersive estimates for the {N}avier-{S}tokes equations in the
  rotational framework.
\newblock {\em Adv. Differential Equations}, 19(9-10):857--878, 2014.

\bibitem[KLT14b]{KohLeeTakada-Str-Euler}
Youngwoo Koh, Sanghyuk Lee, and Ryo Takada.
\newblock Strichartz estimates for the {Euler} equations in the rotational
  framework.
\newblock {\em J. Differ. Equations}, 256(2):707--744, 2014.

\bibitem[KPTW25]{compr-euler-strichartz}
Haram Ko, Benoit Pausader, Ryo Takada, and Klaus Widmayer.
\newblock Increased lifespan for 3d compressible {Euler} flows with rotation.
\newblock {\em arXiv e-prints}, arXiv:2509.20505, 2025.

\bibitem[MB02]{majda2002vorticity}
Andrew~J. Majda and Andrea~L. Bertozzi.
\newblock {\em Vorticity and incompressible flow}.
\newblock Camb. Texts Appl. Math. Cambridge: Cambridge University Press, 2002.

\bibitem[PW18]{Pusateri_2018}
Fabio Pusateri and Klaus Widmayer.
\newblock On the global stability of a beta-plane equation.
\newblock {\em Anal. PDE}, 11(7):1587--1624, 2018.

\bibitem[Saw20]{Y20-BesovSpaces}
Yoshihiro Sawano.
\newblock Homogeneous {Besov} spaces.
\newblock {\em Kyoto J. Math.}, 60(1):1--43, 2020.

\bibitem[Tak19]{TakBQ3D}
Ryo Takada.
\newblock Strongly stratified limit for the {3D} inviscid {Boussinesq}
  equations.
\newblock {\em Arch. Ration. Mech. Anal.}, 232(3):1475--1503, 2019.

\bibitem[Tak21]{TakadaBQ2D}
Ryo Takada.
\newblock Long time solutions for the 2{D} inviscid {B}oussinesq equations with
  strong stratification.
\newblock {\em Manuscripta Math.}, 164(1-2):223--250, 2021.

\bibitem[Tao06]{taononlinear}
Terence Tao.
\newblock {\em Nonlinear dispersive equations. {Local} and global analysis},
  volume 106 of {\em CBMS Reg. Conf. Ser. Math.}
\newblock Providence, RI: American Mathematical Society (AMS), 2006.

\bibitem[Val17]{vallis2017atmospheric}
Geoffrey~K. Vallis.
\newblock {\em Atmospheric and oceanic fluid dynamics. {Fundamentals} and
  large-scale circulation}.
\newblock Cambridge: Cambridge University Press, 2nd edition edition, 2017.

\bibitem[WC16]{wan2016}
Renhui Wan and Jiecheng Chen.
\newblock Global well-posedness for the {2D} dispersive {SQG} equation and
  inviscid {Boussinesq} equations.
\newblock {\em Z. Angew. Math. Phys.}, 67(4):22, 2016.
\newblock Id/No 104.

\bibitem[Ye14]{BQ_LWP_Besov}
Zhuan Ye.
\newblock Blow-up criterion of smooth solutions for the {Boussinesq} equations.
\newblock {\em Nonlinear Anal., Theory Methods Appl., Ser. A, Theory Methods},
  110:97--103, 2014.

\end{thebibliography}

\end{document}